\documentclass[hidelinks,onefignum,onetabnum]{siamart220329}

\usepackage{lipsum}
\usepackage{amsfonts, amssymb}
\usepackage{graphicx}
\usepackage{epstopdf}
\usepackage{algorithm}
\usepackage{algpseudocode}
\usepackage[caption=false]{subfig}
\ifpdf
  \DeclareGraphicsExtensions{.eps,.pdf,.png,.jpg}
\else
  \DeclareGraphicsExtensions{.eps}
\fi

\newsiamremark{remark}{Remark}
\newsiamremark{hypothesis}{Hypothesis}
\crefname{hypothesis}{Hypothesis}{Hypotheses}
\newsiamthm{claim}{Claim}
\newsiamthm{criterion}{Critertion}
\newsiamthm{question}{Question}
\newsiamremark{example}{Example}

\headers{Computing the $\lowercase{p}$-Laplacian eigenpairs}{Chuanyuan Ge and Ouyuan Qin}

\title{Computing the \textit{\textbf{\lowercase{p}}}-Laplacian eigenpairs of signed graphs\thanks{ 2025.01.14
\funding{The first author is supported by the National Key R and D Program of China 2020YFA0713100, the National Natural Science Foundation of China (No. 12031017), and Innovation Program for Quantum Science and Technology 2021ZD0302902.}}}

\author{Chuanyuan Ge\thanks{School of Mathematical Sciences, University of Science and Technology of China, Hefei, China (\email{gechuanyuan@mail.ustc.edu.cn}, \email{ouyuanqin@mail.ustc.edu.cn}).}
\and Ouyuan Qin\footnotemark[2]}

\usepackage{amsopn}

\graphicspath{{../figures/}}

\ifpdf
\hypersetup{
  pdftitle={Largest p-laplacian eigenvalue},
  pdfauthor={Chuanyuan Ge and Ouyuan Qin}
}
\fi

\begin{document}

\maketitle

\begin{abstract}
As a nonlinear extension of the graph Laplacian, the graph $p$-Laplacian has various applications in many fields. Due to the nonlinearity, it is very difficult to compute the eigenvalues and eigenfunctions of graph $p$-Laplacian. In this paper, we establish the equivalence between the graph $p$-Laplacian eigenproblem and the tensor eigenproblem when $p$ is even. Building on this result, algorithms designed for tensor eigenproblems can be adapted to compute the eigenpairs of the graph $p$-Laplacian. For general $p>1$, we give a fast and convergent algorithm to compute the largest eigenvalue and the corresponding eigenfunction of the signless graph $p$-Laplacian. As an application, we provide a new criterion to determine when a graph is not a subgraph of another one, which outperforms existing criteria based on the linear Laplacian and adjacency matrices. Our work highlights the deep connections and numerous similarities between the spectral theories of tensors and graph $p$-Laplacians.
\end{abstract}



\section{Introduction}
The graph $p$-Laplacian introduced by Amghibech is a generalization of the graph Laplacian \cite{Amghibech}. Due to the freedom of $p$, we can choose an appropriate $p$, so that the algorithm based on graph $p$-Laplacian is superior to the classical Laplacian algorithm. For example, it is shown that the spectral clustering algorithm based on graph $p$-Laplacian for $1<p<2$ performs better than classical Laplacian spectral clustering \cite{HeinBuhler2009}. The reason is that, by monotonicity of graph $p$-Laplacian eigenvalues \cite{Zhang23}, the Cheeger inequality given by the second eigenvalue of the graph $p$-Laplacian \cite{Amghibech} is better than that of classical Laplacian of the graph \cite{Alon86,AM85}. In semi-supervised learning, the classic Laplacian algorithm only provides trivial results when the number of unlabeled data points goes to infinity with the number of labeled data points being fixed \cite{NSZ}. In 2019, Slep\v cev and  Matthew proved that this problem could be overcome by choosing $p$ larger than the dimension of the data points \cite{DM}. In addition to the areas mentioned above, the graph $p$-Laplacian  has important applications in many other areas such as combinatorics \cite{BS18}, geometric group theory \cite{De,DM19}, image processing \cite{ETT,LDLLZX}, and machine learning \cite{HeinBuhler2010,ZHS06,ZS06}.

Although the graph $p$-Laplacian has many applications, our understanding of its eigenvalues and eigenfunctions remains limited due to its inherent nonlinearity. Many fundamental problems are still open. For example, it is still not known whether the number of eigenvalues of the graph $p$-Laplacian is finite or not. In addition, computing the eigenvalues and eigenfunctions of the graph $p$-Laplacian is a very difficult problem. In 2009, B\"uhler and Hein used optimization methods to calculate the second eigenvalue and the corresponding eigenfunctions of the graph $p$-Laplacian \cite{HeinBuhler2009}. Luo, Huang, Ding, and Nie gave an efficient gradient descend optimization approach to approximating the $p$-Laplacian embedding space \cite{LHDN}. Deidda, Segala, and Putti provided an algorithm of computing the graph $p$-Laplacian eigenpairs for $p\in(2,+\infty)$ by reformulating the graph $p$-Laplacian eigenvalue problem in terms of a constrained weighted Laplacian eigenvalue problem \cite{DNP}. Their numerical results indicate that the algorithm generally produces acceptable outcomes. However, there are certain scenarios in which the algorithm fails to converge.
 
In this paper, our primary focus is the eigenvalue and eigenfunction calculation problem of $p$-Laplacian on signed graphs. A signed graph is a graph where each edge is assigned a signature, which is a mapping from the edge set of the graph into $\{+1,-1\}$. Signed graphs have received increasing attention since it has played a crucial role in solving many open problems, such as sensitivity conjecture and the problem of the maximum number of spherical two-distance sets \cite{Huang,JTYZZ,JP}. By reformulating the $p$-Laplacian eigenproblems of signed graphs as tensor eigenproblems, we show that the $p$-Laplacian has only finitely many eigenvalues if $p$ is an even integer. Moreover, this reformulation allows us to employ algorithms calculating tensor eigenvalues to compute the signed graph $p$-Laplacian eigenpairs with $p$ being even. Cui, Dai, and Nie \cite{CDN} introduced an algorithm that computes all eigenvalues of the tensor eigenvalues. In \cref{sec:all}, we adapt this algorithm to compute all eigenpairs of $p$-Laplacian of signed $C_4$.

The above algorithm can calculate all eigenpairs of the signed graph $p$-Laplacian when $p$ is even, but it is computationally expensive. For the case where the signs on all edges are fixed as $-1$, we propose an efficient algorithm to compute the largest eigenvalue and the corresponding eigenfunction of the graph $p$-Laplacian for any $p>1$. The convergence of the algorithm is proven with the aid of the Perron--Frobenius type theorem, and numerical results align with our analysis. Since no explicit matrices or tensors are constructed during the algorithm, the solution process is fast, and the method can handle large-scale problems. It is also worth noting that this algorithm reduces to the power method for nonnegative tensor eigenvalue problems \cite{CPZ, NQZ} when $p$ is even.

The largest eigenvalue of the $p$-Laplacian is significant as it has many applications, including those related to forbidden subgraph type problems. A graph $G$ forbids a family of graphs $\mathcal{H}$ if for any $G'\in \mathcal{H}$, $G'$ is not a subgraph of $G$. The problem of determining the properties of a graph that forbids a family of graphs is a fundamental topic in graph theory. The properties include the maximum number of edges \cite{AY,ES1,ES2,MY}, chromatic number \cite{CMZ}, eigenvalue multiplicity \cite{JP,JTYZZ}, and so on. In spectral graph theory, it is known that the largest eigenvalues of the Laplacian, signless Laplacian, and adjacency matrix of $G'$ are smaller than those of a graph $G$ when graph $G'$ is a subgraph of graph $G$. This result can be used to determine whether graph $G'$ cannot be a subgraph of graph $G$. We extend this result to the graph $p$-Laplacian (see \cref{lemma:criterion}). By allowing for the free choice of $p$, we enhance the effectiveness of the above method. Theoretically, we prove that star graph $K_{1,d}$ is a subgraph of $G$ if and only if the largest eigenvalue of signless $p$-Laplacian of $K_{1,d}$ is less than that of $G$ for any $p > 1$. In \cref{sec:app}, we will present an example where the linear Laplacian and adjacency matrix fail, while while the $p$-Laplacian proves to be effective.

Our work reveals that although the eigenproblems of the graph $p$-Laplacian and tensor cannot be directly equivalent when $p$ is not even, they have deeper connections and similarities. We also hope that our work will inspire researchers in the fields of graph $p$-Laplacian or tensors to establish closer connections between these two types of problems and promote the development of these two fields.

The paper is structured as follows. In \cref{sec:pre}, we collect preliminaries on $p$-Laplacians, signed graphs, and tensors. In \cref{sec:all}, we reformulate the signed graph $p$-Laplacian eigenproblem in terms of the tensor eigenproblem and apply an adapted algorithm from tensor theory to compute all $p$-Laplacian eigenpairs of signed graphs. An algorithm to compute the largest eigenvalue and corresponding eigenfunction of signless graph $p$-Laplacian is proposed in \cref{sec:-1}, along with an analysis of its convergence. In \cref{sec:app}, we provide a criterion to determine whether a graph can be a subgraph of another based on the $p$-Laplacian, and demonstrate that this criterion improves upon the corresponding criterion based on the linear Laplacian and adjacency matrix. Concluding remarks and future work are summarized in \cref{sec:remark}.

\section{Preliminaries}\label{sec:pre}
This section reviews basic settings and notations in graph theory and definitions of both graph $p$-Laplacian eigenpairs and tensor eigenpairs.

A weighted graph is a quintuple $G=(V,E,w,\mu,\kappa)$ where $V$ is the vertex set, $E$ is the edge set which consists of binary subset of $V$, $w:E\to (0,\infty)$ is the edge measure, $\mu:V\to (0,\infty)$ is the vertex measure, and $\kappa:V\to\mathbb{R}$ is the potential function. In this paper, we only consider finite simple undirected graphs. We use $i\sim j$ to denote $\{i,j\}\in E$. A signed graph $\Gamma=(G,\sigma)$ is a weighted graph $G$ with a signature $\sigma:E\to \{+1,-1\}$. For ease of notation, we use $w_{ij}$, $\mu_i$, $\kappa_i$, and $\sigma_{ij}$ to denote $w(\{i,j\})$, $\mu(i)$, $\kappa(i)$, and $\sigma(\{i,j\})$, respectively.   Without loss of generality, we always label $V=\{1,2\ldots n\}$.

Given two weighted graph $G=(V,E,w,\mu,\kappa)$ and $G'=(V',E',w',\mu',\kappa')$, we say $G'$ is a subgraph of $G$ if $V'\subset V$, $E'\subset E$, $w' = w|_{E'}$, $\mu' = \mu|_{V'}$, and $\kappa' = \kappa|_{V'}$.

We use $C(V)$ to denote the set of all real functions on $V$, and we always identify $C(V)$ with $\mathbb{R}^n$. Given any $f\in \mathbb{R}^n$, we use $f_i$ to denote the $i$-th coordinate of $f$.

We define $\mathbb{P}^n:=\{f\in \mathbb{R}^n:f_i\geq 0 \text{ for }1\leq i\leq n\}$ and $\mathrm{int}(\mathbb{P}^n):=\{f\in \mathbb{R}^n:f_i> 0 \text{ for }1\leq i\leq n\}$. For any $f\in \mathbb{R}^n$ and $t>0$, we define 
\begin{align*}
  f^{t}:=\left(|f_1|^{t}\mathrm{sign}(f_1),|f_2|^{t}\mathrm{sign}(f_2),\ldots,|f_n|^{t}\mathrm{sign}(f_n)\right).
\end{align*}
For any $f,g\in \mathbb{R}^n$, we say $f\geq g$ if  $f-g\in \mathbb{P}^n$ and $f>g$ if $f-g\in \mathrm{int}(\mathbb{P}^n)$.

In this paper, when we use the constant $p\in\mathbb{R}$, we assume that $p>1$. Define $\Phi_p:\mathbb{R}\to \mathbb{R}$ to be  $\Phi_p(t)=|t|^{p-2}t$ if $t\neq 0$ and  $\Phi_p(t)=0$ if $t=0$.

\begin{definition}
  The $p$-Laplacian $\Delta^{\Gamma}_p:C(V)\to C(V)$ of a signed graph $\Gamma$ is defined by
  \begin{align*}
    \left(\Delta^{\Gamma}_p f\right)_i=\frac{1}{\mu_i}\left(\sum_{j\sim  i}w_{ij}\Phi_p\left(f_i-\sigma_{ij}f_j\right)+\kappa_i\Phi_p(f_i)\right),\;\;\forall i\in V,\;\;\forall f\in C(V).
  \end{align*}
If it is clear, we use $\Delta_p $ to denote $\Delta^{\Gamma}_p$.
\end{definition}

Let $G=(V,E,w,\mu,\kappa)$ be a weighted graph. The $p$-Laplacian of the graph $G$ is defined as the $p$-Laplacian of signed graph $\Gamma=(G,\sigma)$ with $\sigma\equiv+1$. The signless $p$-Laplacian of $G$ is defined as the $p$-Laplacian of signed graph $\overline{\Gamma}=(G,\overline{\sigma})$ with $\overline{\sigma}\equiv-1$.

\begin{definition}
  Given $\lambda\in \mathbb{R}$, a nonzero function $f:V\to \mathbb{R}$ is an eigenfunction of $\Delta_p$ corresponding to $\lambda$ if 
  \begin{align*}
    \left(\Delta_pf\right)_i=\lambda \Phi_p\left(f_i\right),\quad\quad\forall i\in V.
  \end{align*}
We call $\lambda$ is an eigenvalue of $\Delta_p$ and $(\lambda,f)$ is an eigenpair of $\Delta_p$.
\end{definition}

\begin{remark}
  We can extend the eigenproblem to complex cases by extending the domain of $\Phi_p$ to the complex plane. However, this is less interesting because it is not difficult to prove that $\Delta_p$ only has real eigenvalues in this situation and every eigenvalue has at least one real eigenfunction corresponding to it.
\end{remark}

For any $f\in \mathbb{R}^n$, we define $\Vert f \Vert_{p}:=\left(\sum_{i=1}^n\mu_i|f_i|^p\right)^{\frac{1}{p}}$ and $\mathcal{S}_p:=\{f\in \mathbb{R}^n:\Vert f \Vert_p=1\}$. We always use $\Vert f \Vert$ to denote $\Vert f \Vert_2$. We define the Rayleigh quotient of a nonzero function $f$ corresponding to $\Delta_p^{\Gamma}$ as follows 
\begin{align*}
  \mathcal{R}_p^{\Gamma}(f):=\frac{\sum_{\{i,j\}\in E}w_{ij}\left|f_i-\sigma_{ij}f_j\right|^p+\sum_{i=1}^{n}\kappa_i|f_i|^p}{\sum_{i=1}^n\mu_i|f_i|^p}.
\end{align*}
If no confusion may occur, we use $\mathcal{R}_p(f)$ to denote $\mathcal{R}_p^{\Gamma}(f)$.

 The largest eigenvalue is the eigenvalue whose value is the largest in all eigenvalues. We denote  the largest eigenvalue by $\lambda_{\max}(\Delta_p)$. If no confusion arises, we use $\lambda_{\max}$ instead of $\lambda_{\max}(\Delta_p)$ for ease of notation.
 
The extension of  the largest eigenvalue is followed by the Lusternik--Schnirelman theory \cite{GLZ23,Solimini,Struwe,TudiscoHein18}. Moreover, it gives a variational expression of 
the largest eigenvalue as follows:
\begin{align*}
  \lambda_{\max}(\Delta_p)=\max_{f\in \mathcal{S}_p}\mathcal{R}_p(f).
\end{align*}
 For other variational eigenvalues, their expression can be found in other literature, such as \cite{chang16,DPT23,GLZ24,GLZ23,TudiscoHein18}.

We recall the definition of tensors and tensor eigenvalues. An $m$-th order $n$-dimensional tensor $\mathcal{T}$ consists of $n^m$  entries $\mathcal{T}_{i_1,\ldots,i_m}\in \mathbb{R}$, where $i_l\in\{1,\ldots,n\}$ for $l=1,\ldots,m$. A tensor $\mathcal{T}$ is called symmetric if its entries are invariant under any permutation of their indices. 

For a $m$-th order $n$-dimensional tensor $\mathcal{T}$, we define an operator $\mathcal{T}:\mathbb{R}^n\to\mathbb{R}^n $ as follows:
\begin{align*}
  \left(\mathcal{T}f\right)_i=\sum_{i_2,\ldots,i_m=1}^n\mathcal{T}_{i,i_2,\ldots,i_m}f_{i_2}\cdots f_{i_m}, \text{ for any } f\in \mathbb{R}^n \text{ and } 1\leq i \leq n.
\end{align*}

\begin{definition}[\cite{CDN}]
  Let $\mathcal{T}$ and $\mathcal{B}$ be two $m$-th order $n$-dimensional tensors. Given $\eta\in \mathbb{R}$, a non-zero function $f$ is a  $\mathcal{B}$-eigenfunction of $\mathcal{T}$ corresponding to $\eta$ if  \[\mathcal{T}f=\eta \mathcal{B}f.\]
  We call $\eta$ a $\mathcal{B}$-eigenvalue of $\mathcal{T}$ and  $(\eta,f)$ a $\mathcal{B}$-eigenpair of $\mathcal{T}$. 
\end{definition}

\begin{remark}
  With different $\mathcal{B}$, the definition above degenerates into specific types of tensor eigenvalue, such as H-eigenvalue, D-eigenvalue, and Z-eigenvalue \cite{Lim15,Qi05,QCC}.
\end{remark}

\section{Equivalence of signed graph \texorpdfstring{$p$-Laplacian eigenproblem and tensor eigenproblem for even $p$}{p-Laplacian eigenproblem and tensor eigenproblem for even p}}\label{sec:all}
In this section, we transform the problem of eigenpairs of signed graph $p$-Laplacian into the problem of eigenpairs of $p$-th order tensor when $p$ is even, and compute all eigenpairs of the $p$-Laplacian in virtue of algorithms designed for tensor eigenpairs.

First, we show the connection between the tensor eigenpairs and the eigenpairs of signed graph $p$-Laplacian in the following proposition, which is an improved result of Proposition 7.1 in \cite{GLZ24}.
\begin{proposition}\label{prop:tensor}
  Let $\Gamma=(G,\sigma)$ be a signed graph with $G=(V,E,w,\mu,\kappa)$. Assume $p \geq 2$ is an even number and $i$ and $j$ are arbitrary vertices in $V$. Define two symmetric $p$-th order $n$-dimension tensors $\mathcal{T}^{(p)}$ and $\mathcal{B}^{(p)}$ as follows
\begin{align}\label{eq:tensor}
  \mathcal{T}^{(p)}_{i_1,i_2,\ldots,i_{p}}=\left\{ \begin{aligned}
    &0, && |\{i_1,i_2,\ldots,i_{p}\}|\geq 3,\\
    &0, && \{i_1,i_2,\ldots,i_{p}\}=\{i^{(l)},j^{(p-l)}\} \text{ and } i \nsim j,\\
    &(-\sigma_{ij})^{p-l}w_{ij}, && \{i_1,i_2,\ldots,i_{p}\}=\{i^{(l)},j^{(p-l)}\} \text{ and }i \sim j,\\
    &\sum_{j \sim i}w_{ij}+\kappa_{i}, && \{i_1,i_2,\ldots,i_{p}\}=\{i^{(p)}\},
  \end{aligned}
  \right.
\end{align}  
and
\begin{align}\label{eq:tensor2}
  \mathcal{B}^{(p)}_{i_1,i_2,\ldots,i_{p}}=\left\{ \begin{aligned}
    &\mu_i, && i_1=i_2=\cdots=i_p = i ,\\
    &0,  && \text{otherwise},
  \end{aligned}
  \right.
\end{align}   
where $|\{i_1,i_2,\ldots,i_{p}\}|$ is the cardinality of set $\{i_1,i_2,\ldots,i_{p}\}$ and $\{i_1,i_2,\ldots,i_{p}\}=\{i^{(l)},j^{(p-l)}\}$ means there are $l$ $i$'s and $(p-l)$ $j$'s in $\{i_1,i_2,\ldots,i_{p}\}$. Then $(\lambda,f)$ is an eigenpair of $\Delta_p$ if and only if it is a $\mathcal{B}^{(p)}$-eigenpair of $\mathcal{T}$.
\end{proposition}
\begin{proof}
  Given any $f:V\to \mathbb{R}$ and $1 \leq i \leq n$, we compute
\begin{align*}
  \begin{aligned}   
    \left(\mathcal{T}^{(p)}f\right)_i&=\sum_{i_2,\ldots,i_p=1}^n\mathcal{T}^{(p)}_{i,i_2,\ldots,i_p}f_{i_2}\cdots f_{i_p}\\
    &=\bigg(\kappa_i+\sum_{j\sim i}w_{ij}\bigg)f_i^{p-1}+\sum_{j\sim i}\sum_{l=1}^{p-1}\tbinom{p-1}{l}(-\sigma_{ij})^lw_{ij}f_i^{p-1-l}f_j^{l}\\
    &=\sum_{j\sim i}w_{ij}\left(f_i-\sigma_{ij}f_j \right)^{p-1}+\kappa_if_i^{p-1}\\
    &=\mu_i\left(\Delta_{p} f\right)_i,
  \end{aligned}
\end{align*}
and $\left(\mathcal{B}^{(p)} f\right)_i=\mu_i\Phi_p(f_i)$. 

Then we have $\left(\Delta_p f\right)_i=\lambda\Phi_p(f_i)$ if and only if $\left(\mathcal{T}^{(p)}f\right)_i=\lambda\left(\mathcal{B}^{(p)} f\right)_i$ for any $1\leq i \leq n$. This completes the proof of this proposition. 
\end{proof}

Due to the above proposition, we have that the problem of eigenpairs of signed graph $p$-Laplacian is equivalent to the problem of eigenpairs of the tensor.

It is an open problem whether the number of the eigenvalues of signed graph $p$-Laplacian is finite. The above transformation allows us to use existing results to demonstrate that, in the case of even $p$, the signed graph $p$-Laplacian has only a finite number of eigenvalues. Actually, in the literature of \cite{CDN}, the authors showed the following fact.
\begin{lemma}
  Given two symmetric tensors $\mathcal{T}$ and  $\mathcal{B}$, there are finitely many $\mathcal{B}$-eigenvalues of $\mathcal{T}$.
\end{lemma}

We have the following fact by the above lemma and \cref{prop:tensor}.
\begin{theorem}
  If $p$ is even, we have the number of eigenvalues of $\Delta_p$ is finite.
\end{theorem}

According to \cref{prop:tensor}, we can apply any algorithms designed to compute tensor eigenpairs, such as the shifted power method \cite{tam} and the Newton correction method \cite{ari}, to compute the eigenpairs of the signed graph $p$-Laplacian, provided $p$ is even. We adapt, however, the algorithm in \cite{CDN} as an approach to computing eigenpairs of $p$-Laplacian for its capabilities to compute all eigenpairs simultaneously.
 \begin{example}\label{ex:nodal}
  Let $\Gamma=(G,\sigma)$  be defined as follows, where $G=(V,E,w,\mu,\kappa)$ and
  \begin{align*}
    \left\{
    \begin{aligned}
      &V=\{1,2,3,4\}, ~ E=\big\{\{1,2\},\{2,3\},\{3,4\},\{1,4\}\big\},\\
      &  w_{12}=w_{23}=w_{34}=1 \text{ and } w_{14}=2,\\
      &  \sigma_{12}=\sigma_{23}=\sigma_{34}=+1 \text{ and } \sigma_{14}=-1,\\
      &  \mu_2=\mu_3=\mu_4=1  \text{ and } \mu_{1}=2,\\
      &  \kappa_1=\kappa_2=\kappa_3=1  \text{ and } \kappa_{4}=2.\\
    \end{aligned}
    \right.
\end{align*}
The total 12 eigenpairs of $\Delta_p$ are shown in \cref{tab:k4} when $p = 4$.
\end{example}

\begin{table}[!t]
  \centering
  \caption{All eigenpairs of $\Delta_4$ for signed $C_4$.}
  \begin{tabular}{|c|c|c|c|c|c|c|}
    \hline
    No. & 1 & 2 & 3 & 4 & 5 & 6 \\
    \hline
    $\lambda$ & 1.686e1 & 1.668e1 & 1.631e1 & 1.463e1 & 1.448e1 & 1.369e1 \\
    \hline
    $f_1$ & 5.568e-1 & 6.186e-1 & 5.952e-1 & 6.693e-1 & 6.796e-1 & -3.717e-1 \\
    $f_2$ & 4.556e-1 & -4.107e-1 & -1.066e-2 & -6.355e-1 & -5.469e-1 & 7.471e-1 \\
    $f_3$ & -6.894e-1 & -5.780e-1 & -6.164e-1 & 4.493e-1 & 1.644e-1 & -8.258e-1 \\
    $f_4$ & 8.567e-1 & 8.678e-1 & 8.818e-1 & 7.927e-1 & 8.337e-1 & 6.560e-1 \\
    \hline
    No. & 7 & 8 & 9 & 10 & 11 & 12 \\
    \hline
    $\lambda$ & 1.344e1 & 1.162e1 & 1.608 & 1.364 & 1.301 & 7.047e-1 \\
    \hline
    $f_1$ & -1.032e-1 & 5.070e-1 & -5.715e-1 & 2.339e-1 & 2.953e-3 & 8.298e-1 \\
    $f_2$ & 6.907e-1 & -8.595e-1 & 1.737e-1 & 8.174e-1 & 5.181e-1 & 3.946e-1 \\
    $f_3$ & -8.416e-1 & 7.533e-1 & 9.169e-1 & 8.588e-1 & 9.742e-1 & -6.080e-3 \\
    $f_4$ & 7.211e-1 & 5.107e-2 & 5.296e-1 & 2.459e-1 & 4.060e-1 & -4.067e-1 \\
    \hline
  \end{tabular}
  \label{tab:k4}
\end{table}

\begin{remark}
  For the case of linear Laplacian, given any two eigenpairs $(\lambda,f)$ and $(\mu,g)$, we have $\lambda=\mu$ if $f_i g_i>0$ for any $i\in V$. Note that this conclusion does not hold anymore for general $p$-Laplacian. From \cref{ex:nodal}, we have $f^{(2)}_i f^{(3)}_i>0$ for any $i\in V$ but $\lambda_2 \neq \lambda_3$, where $f^{(2)}$ and $f^{(3)}$ are the eigenfunctions corresponding to $\lambda_2$ and $\lambda_3$, respectively. 
\end{remark}

\section{An algorithm based on Perron--Frobenius type theorem}\label{sec:-1}
In this section, we propose an algorithm for computing the largest eigenvalue and corresponding eigenfunction of the $p$-Laplacian of $\Gamma=(G,\sigma)$ with $\sigma\equiv-1$, and prove its convergence.

Given any signed graph $\Gamma=(G,\sigma)$ with $G=(V,E,w,\mu,\kappa)$, if there exists $i\in V$, such that $\kappa_i < 0$, we define $c=\max_{i\in V}\left|\frac{\kappa_i}{\mu_i}\right|$ and a new signed graph $\Gamma'=(G',\sigma')$ with  $G'=(V,E,w,\mu,\kappa')$ where $\kappa'_i=\kappa_i+c\mu_i$ and $\sigma'\equiv\sigma.$ By direct computation, we have $\kappa'_i\geq 0$ for any $i\in V$ and $\lambda_{\max}(\Delta_p^{\Gamma})=\lambda_{\max}(\Delta_p^{\Gamma'})-c$. Therefore, if an algorithm can successfully compute the largest eigenvalue and the corresponding eigenfunction of the $p$-Laplacian for signed graphs with a nonnegative potential function, it can also be applied to solve the general case of signed graphs. Without loss of generality, we assume $\kappa_i\geq 0$ for any $i\in V$ in the following analysis of the algorithm.


The following Perron--Frobenius type theorem plays an important role in our analysis \cite{GLZ23}.

\begin{theorem}\label{thm:pr}
Let $\Gamma=(G,\sigma)$  be a signed graph with $G=(V,E,w,\mu,\kappa)$ and  $\sigma\equiv -1$. Assume $G$ is connected.  Given an eigenfunction $f$ of $\Delta_p$ corresponding to the largest eigenvalue $\lambda_{\max}$, we have the following properties:
	\begin{itemize}
    \item [(i)] Either $f_i>0$ for any $i\in V$ or $f_i<0$ for any $i\in V$;
    \item [(ii)] If $g$ is an eigenfunction of $\Delta_p$ corresponding to $\lambda_{\max}$, then there exists a constant $c\in \mathbb{R}\setminus \{0\}$ such that $g=cf$;	 
    \item [(iii)]If $g$ is an eigenfunction of $\Delta_p$ corresponding to an eigenvalue $\lambda$ with $g_i> 0$ or $g_i< 0$ for any $i\in V$, then $\lambda=\lambda_{\max}$.
	\end{itemize}
\end{theorem}

\cref{thm:pr} $(iii)$ can be improved. Exactly, we have the following lemma.
\begin{lemma}\label{lemma:pr}
  Let $\Gamma$, $f$, $\lambda_{\max}$ be defined as in \cref{thm:pr} and $(\lambda,g)$ be an eigenpair of $\Delta_p$. If $g_i\geq 0$ for any $i\in V$ or $g_i\leq 0$ for any $i\in V$, we have $\lambda=\lambda_{\max}$.
\end{lemma}
\begin{proof}
  We assume $g_i\geq 0$  for any $i\in V$. Assume there exists $i_0\in V$ such that $g_{i_0}=0$. By definition of eigenpair, we have $\sum_{j\sim i_0}w_{i_0j}\Phi_p\left(g_j\right)=0.$ Since $g_i\geq 0$ for any $i\in V$, we have $g_j=0$ for any $j\sim i_0$. Then we have $g\equiv0$ by induction. This contradicts the definition of eigenpair. This implies $g_i> 0$  for any $i\in V$. By \cref{thm:pr} $(iii)$, we have $\lambda=\lambda_{\max}$. The proof of the case $g_i\leq 0$ for any $i\in V$ is the same.
\end{proof}

The following lemma, which can be derived through direct computation, is instrumental for our analysis.
\begin{lemma}\label{lemma:basic}
  Let $\Gamma=(G,\sigma)$ be a signed graph with $G=(V,E,w,\mu,\kappa)$, $\sigma\equiv-1$ and $\kappa_i\geq 0$ for any $i\in V$. If $f,g\in \mathbb{P}^n$ satisfy $f\leq g$, we have $$\Delta_p(f)\leq \Delta_p(g).$$
\end{lemma}

Let $f$ be an eigenfunction of $\Delta_p$ corresponding to $\lambda_{\max}$ in the following of this section. Moreover, we assume  $\Vert f\Vert=1$ and $f\in \mathrm{int}(\mathbb{P}^n)$. The reason that $f\in \mathrm{int}(\mathbb{P}^n)$ is justified by \cref{thm:pr}.

We fix any $f^{(0)}\in \mathrm{int}(\mathbb{P}^n)$ and define $g^{(0)}=\Delta_pf^{(0)}$. For any $k\in \mathbb{Z}^+$, we inductively define
$$f^{(k)}=\frac{(g^{(k-1)})^{\frac{1}{p-1}}}{\Vert (g^{(k-1)})^{\frac{1}{p-1}}\Vert},\quad\quad\quad g^{(k)}=\Delta_pf^{(k)},$$
and
\begin{align}
\label{minmaxlambda}
\underline{\lambda}_{k}=\min_{i\in V}\frac{\left(\Delta_pf^{(k)}\right)_i}{\left(f_i^{(k)}\right)^{p-1}},\quad\quad\quad\overline{\lambda}_{k}=\max_{i\in V}\frac{\left(\Delta_pf^{(k)}\right)_i}{\left(f_i^{(k)}\right)^{p-1}}.
\end{align}
It is worth noting that for any $k$, $f^{(k)}$ and $g^{(k)}$ are both in $\mathrm{int}(\mathbb{P}^n)$ by directly computing. This implies both $\underline{\lambda}_{k}$ and $\overline{\lambda}_{k}$ are well-defined.

\begin{theorem}\label{thm:main}
For any integer $k\geq 1$, we have 
\begin{itemize}
          \item [(i)] $\underline{\lambda}_{k}\leq \underline{\lambda}_{k+1}$ and $\overline{\lambda}_{k+1}\leq \overline{\lambda}_{k}$,
          \item [(ii)]  $\underline{\lambda}_{k}\leq \lambda_{\max}\leq \overline{\lambda}_{k}$,
	   \item [(iii)] $f^{(k)}$ are convergent and $$\lim_{k\to \infty}f^{(k)}=f,\quad\quad\quad\lim_{k\to\infty}\underline{\lambda}_{k}=\lim_{k\to\infty}\overline{\lambda}_{k}=\lambda_{\max}.$$
	\end{itemize}
\end{theorem}
We only prove $\underline{\lambda}_{k}\leq \underline{\lambda}_{k+1}\leq \lambda_{\max}$ and $\lim_{k\to \infty}\underline{\lambda}_{k}=\lambda_{\max}$, since the same techniques can be exploited to prove the case of $\overline{\lambda}_{k}$.

\begin{proof}[Proof of $(i)$ in \cref{thm:main}]
   By definition, for any $i\in V$, we have $$g_i^{(k)}= \left(\Delta_pf^{(k)}\right)_i\geq\underline{\lambda}_{k}\left(f_i^{(k)}\right)^{p-1}.$$ This implies
   $$\left(g_i^{(k)}\right)^{\frac{1}{p-1}}\geq \left(\underline{\lambda}_{k}\right)^{\frac{1}{p-1}}f_i^{(k)},$$
and $$f_i^{(k+1)}=\frac{\left(g_i^{(k)}\right)^{\frac{1}{p-1}}}{\Vert \left(g^{(k)}\right)^{\frac{1}{p-1}} \Vert}\geq \frac{\left(\underline{\lambda}_{k}\right)^{\frac{1}{p-1}}f_i^{(k)}}{\Vert \left(g^{(k)}\right)^{\frac{1}{p-1}} \Vert}.$$
Then we have $$f^{k+1}\geq \frac{\left(\underline{\lambda}_{k}\right)^{\frac{1}{p-1}}f^{(k)}}{\Vert \left(g^{(k)}\right)^{\frac{1}{p-1}} \Vert}.$$
For any $j\in V$, we compute 
\begin{align*}
  \left(\Delta_pf^{(k+1)}\right)_j \geq \frac{\underline{\lambda}_{k}\left(\Delta_p f^{(k)}\right)_j}{\Vert \left(g^{(k)}\right)^{\frac{1}{p-1}} \Vert^{p-1}}
  =\frac{\underline{\lambda}_{k}\left[\left(g^{(k)}_j\right)^\frac{1}{p-1}\right]^{p-1}}{\Vert \left(g^{(k)}\right)^{\frac{1}{p-1}} \Vert^{p-1}}
  =\underline{\lambda}_{k}\left(f_j^{(k+1)}\right)^{p-1},
\end{align*}
where the first inequality is by \cref{lemma:basic}. This implies that $\underline{\lambda}_{k+1}\geq \underline{\lambda}_{k}$, and the assertion in (i) follows.  
\end{proof}

To prove (ii), we need the following lemma.
\begin{lemma}\label{lemma:com}
    If there exist $g\in \mathbb{P}^n\setminus \{0\}$ such that $\Delta_pg\geq  \eta g^{p-1}$, we have $\eta\leq \lambda_{\max}$.
\end{lemma}
\begin{proof}
Define $t_0=\max\{t\in \mathbb{R}:f-tg\geq 0\}$. Since $f$ is positive on every vertices, we have $t_0>0$. For any $i\in V$, we compute $$\lambda_{\max} f_i^{p-1}=(\Delta_pf)_i\geq t_0^{p-1}(\Delta_pg)_i\geq t_0^{p-1}\eta g_i^{p-1}.$$ This shows $$f\geq t_0\left(\frac{\eta}{\lambda_{\max}}\right)^{\frac{1}{p-1}}g.$$ By definition of $t_0$, we have $\eta \leq \lambda_{\max}$. This concludes the proof of this lemma.
\end{proof}

Then we prove (ii) in \cref{thm:main}.
\begin{proof}[Proof of (ii) in \cref{thm:main}] 
  By definition, we have $\Delta_pf^{(k)}\geq \underline{\lambda}_{k}\left(f^{(k)}\right)^{p-1}$. By \cref{lemma:com}, we have $\underline{\lambda}_{k}\leq \lambda$.
\end{proof}

For the proof of (iii) in \cref{thm:main}, we define $\mathcal{D}:\mathbb{P}^n \to \mathbb{P}^n$ by 
$$\mathcal{D}f:=\left(\Delta_pf\right)^{\frac{1}{p-1}}, \text{ for any $f\in \mathbb{P}^n$.}$$ We have the following lemma.
\begin{lemma}\label{lemma:increasing}
  For any $f,g\in \mathbb{P}^n$ with $f\leq g$ and $f\neq g$, there exists $r_0$ such that $\mathcal{D}^{r_0}f<\mathcal{D}^{r_0}g$. 
\end{lemma}
\begin{proof}
Since $f\leq g$, we have $\mathcal{D}^{k}f\leq \mathcal{D}^{k}g$ for any $k$ by Lemma \ref{lemma:basic} and induction. Because $f\neq g$, we assume that $f_1<g_1$ without loss of generality. Define $$r_0=\max_{i\in V}d(i,1),$$
where $d(x,y)$ is the combinatorial distance of the graph $G=(V,E,w,\mu,\kappa)$, i.e., the minimal number of edges in all paths connecting $x$ and $y$. 

\begin{claim}
  \label{clm:distance}
  $\left(\mathcal{D}^kf\right)_j<\left(\mathcal{D}^kg\right)_j$ for any $j\in B_k(1)$, where $B_k(1):=\{i\in V:d(1,i)\leq k\}$.
\end{claim}

This lemma is implied by this claim because we have $B_{r_0}(1)=V$ and $\left(\mathcal{D}^{r_0}f\right)_j<\left(\mathcal{D}^{r_0}g\right)_j$ for any $j\in B_{r_0}(1)$. 

We prove \cref{clm:distance} by induction. We directly have that the claim holds for $k=0$. We assume the claim holds for $k-1$ and prove this claim holds for $k$. For any $j\in B_{k}(1)$, we compute \begin{equation*}
  \begin{aligned}    \left(\mathcal{D}^{k}g\right)_j&=\frac{1}{\mu_j^{\frac{1}{p-1}}}\left(\sum_{l\sim j}w_{lj}\left|\left(\mathcal{D}^{k-1}g\right)_l+\left(\mathcal{D}^{k-1}g\right)_j\right|^{p-1}+\kappa_j\left|\left(\mathcal{D}^{k-1}g\right)_j\right|^{p-1}\right)^{\frac{1}{p-1}}\\
  &>\frac{1}{\mu_j^{\frac{1}{p-1}}}\left(\sum_{l\sim j}w_{lj}\left|\left(\mathcal{D}^{k-1}f\right)_l+\left(\mathcal{D}^{k-1}f\right)_j\right|^{p-1}+\kappa_j\left|\left(\mathcal{D}^{k-1}f\right)_j\right|^{p-1}\right)^{\frac{1}{p-1}}\\
  &=\left(\mathcal{D}^{k}f\right)_j.
     \end{aligned}
\end{equation*}
The above inequality is obtained by induction. This proves the claim.
\end{proof}

\begin{proof}[Proof of (iii) in \cref{thm:main}]
  Since $\{f^{(k)}\}_{k=1}^{\infty}\subset \mathcal{S}_2$ and $\mathcal{S}_2$ is compact, we have that there exists a convergent subsequence $\{f^{(k_j)}\}_{j=1}^{\infty}$ of  $\{f^{(k)}\}_{k=1}^{\infty}$ and denote its limit by $f^*.$ By definition, we have $\underline{\lambda}_{k}^{\frac{1}{p-1}}f_i^{(k)}\leq \left(\mathcal{D}f^{(k)}\right)_i$ for any  $i\in V$ and $k\in \mathbb{N}$. Because $\lim_{j\to \infty}f^{(k_j)}=f^*$, $\underline{\lambda}_{k}\leq \underline{\lambda}_{k+1}$, and $\mathcal{D}$ is continuous on $P$, we have $\underline{\lambda}_{k}^{\frac{1}{p-1}}f_i^*\leq \left(\mathcal{D}f^*\right)_i$ for any $i\in V$. Moreover, because $\underline{\lambda}_{k}\leq \underline{\lambda}_{k+1}\leq \lambda_{\max}$, we have the limit of $\underline{\lambda}_{k}$ exist as $k$ tends to infinity and denote this limit by $\underline{\lambda}$, i.e., $\lim_{k\to \infty}\underline{\lambda}_{k}=\underline{\lambda}$. We have $\underline{\lambda}^{\frac{1}{p-1}}f_i^*\leq \left(\mathcal{D}f^*\right)_i$ for $i\in V$.

  \begin{claim}
    \label{clm:dmf}
    For any $m\in \mathbb{N}$, there exists $i_0\in V$ such that $$\left(\mathcal{D}^{m+1}f^*\right)_{i_0}=\underline{\lambda}^{\frac{1}{p-1}}\left(\mathcal{D}^mf^*\right)_{i_0}.$$
  \end{claim}

  We prove \cref{clm:dmf} by contradiction. Assume there exists $m_0\in \mathbb{N}$, such that $$\left(\mathcal{D}^{m_0+1}f^*\right)_i>\underline{\lambda}^{\frac{1}{p-1}}\left(\mathcal{D}^{m_0}f^*\right)_i,\quad\text{ for any $i\in V$}.$$ By continuity of $\mathcal{D}$, we have that there exists $k_0$ which is large enough such that $$\left(\mathcal{D}^{m_0+1}f^{(k_0)}\right)_i>\underline{\lambda}^{\frac{1}{p-1}}\left(\mathcal{D}^{m_0}f^{(k_0)}\right)_i,\quad\text{ for any $i\in V$.}$$ We compute 
  \begin{align*}
    \begin{aligned}
        \mathcal{D}^{m_0}f^{(k_0)}&=\mathcal{D}^{m_0-1}\left(\Delta_pf^{(k_0)}\right)^{\frac{1}{p-1}}=\mathcal{D}^{m_0-1}\left(g^{(k_0)}\right)^{\frac{1}{p-1}}\\
        &=\mathcal{D}^{m_0-1}\left(\Vert g^{(k_0)}\Vert^{\frac{1}{p-1}}f^{(k_0+1)}\right)\\
        &=\Vert g^{(k_0)}\Vert^{\frac{1}{p-1}}\mathcal{D}^{m_0-1}\left(f^{(k_0+1)}\right)\\
        &=\Vert g^{(k_0)}\Vert^{\frac{1}{p-1}}\Vert g^{(k_0+1)}\Vert^{\frac{1}{p-1}}\cdots \Vert g^{(k_0+m_0-1)}\Vert^{\frac{1}{p-1}}f^{(k_0+m_0)}.
    \end{aligned}
  \end{align*}
  The above fourth equality, we use the fact that $\mathcal{D}(cf)=c\mathcal{D}(f)$ for any $c\geq 0 $ and $f\in \mathbb{P}^n$.  We compute $$\mathcal{D}^{m_0+1}(f^{(k_0)})=\Vert g^{(k_0)}\Vert^{\frac{1}{p-1}}\Vert g^{(k_0+1)}\Vert^{\frac{1}{p-1}}\cdots \Vert g^{(k_0+m_0-1)}\Vert^{\frac{1}{p-1}}\left(g^{(k_0+m_0)}\right)^{\frac{1}{p-1}}.$$
  Combine the above two equalities, we have $$\underline{\lambda}<\min_{i\in V}\frac{\left(\mathcal{D}^{m_0+1}(f^{k_0})\right)_i^{p-1}}{\left(\mathcal{D}^{m_0}(f^{k_0})\right)_i^{p-1}}=\min_{i\in V}\frac{\left(g^{(k_0+m_0)}\right)_i}{\left(f^{(k_0+m_0)}_i\right)^{p-1}}= \underline{\lambda_{k_0+m_0}}\leq \underline{\lambda}.$$
  Contradiction. This concludes the proof of the claim. 

  \begin{claim}
    \label{clm:df}
    $\underline{\lambda}^{\frac{1}{p-1}}f^*= \mathcal{D}f^*.$
  \end{claim}

  We prove \cref{clm:df} by contradiction.  Assume there exist $i_0$ such that $\underline{\lambda}^{\frac{1}{p-1}}f_{i_0}^*<\left(\mathcal{D}f^*\right)_{i_0}$. By \cref{lemma:increasing}, we have there exists $r_0$ such that $\underline{\lambda}^{\frac{1}{p-1}}\mathcal{D}^{r_0}f^*<\mathcal{D}^{r_0+1}f^*$. This contradicts \cref{clm:dmf}. This concludes the proof of \cref{clm:df}.

  Finally, since $\underline{\lambda}f^*= \Delta_pf^*$ and $f\in \mathbb{P}^n$, we obtain $f^*=f$ and $\underline{\lambda}=\lambda$ by \cref{thm:pr} and \cref{lemma:pr}, which proves (iii) in \cref{thm:main}.
\end{proof}

\begin{algorithm}[!t]
  \caption{An algorithm for computing the largest eigenpair of signless graph $p$-Laplacian}
  \label{alg:pm}
  \begin{description}
    \item[Inputs:] A connected signless graph $\Gamma = (G, \sigma)$ with $G=(V,E,w,\mu,\kappa)$ and $\sigma \equiv -1$, the parameter $p$, an initial iteration for eigenvectors $f^{(0)} \in \mathrm{int}(\mathbb{P}^n)$ and a tolerance $0 < \epsilon < 1$.
    \item[Outputs:] The largest eigenvalue $\lambda_{\max}$ for $\Delta^{\Gamma}_p$ and its corresponding eigenvector $f_{\max}$.
  \end{description}
  \hrule
  \begin{algorithmic}[1]
    \If{$\kappa < 0$}
      \State $c=\max_{i \in V} \left|\frac{\kappa_i}{\mu_i}\right|$
      \State Modify $\kappa_i$ as $\kappa_i + c \mu_i$ for each $i$.
    \Else
      \State $c=0$
    \EndIf
    \State $g^{(0)} = \Delta^{\Gamma}_p f^{(0)}$
    \For{$k=1, 2, \dots$}
      \State Compute $f^{(k)} = \frac{(g^{(k-1)})^{\frac{1}{p-1}}}{\| (g^{(k-1)})^{\frac{1}{p-1}} \|}$ and $g^{(k)} = \Delta^{\Gamma}_p f^{(k)}$.
      \State Compute $\underline{\lambda}_{k}$ and $\overline{\lambda}_{k}$ as in \eqref{minmaxlambda}.
      \If{$\left(\overline{\lambda}_{k} - \underline{\lambda}_{k}\right) / \left(\overline{\lambda}_{k} + \underline{\lambda}_{k}\right) < \epsilon$}
        \State $\lambda_{\max} = \left(\overline{\lambda}_{k} + \underline{\lambda}_{k}\right) / 2 - c$
        \State $f_{\max} = f^{(k)}$
        \State \Return $\lambda_{\max}, ~f_{\max}$
      \EndIf
    \EndFor
  \end{algorithmic}
\end{algorithm}

The convergent algorithm analyzed above is summarized in \cref{alg:pm}, which can calculate the largest eigenpair of signless graph $p$-Laplacian. Only two vectors, $f^{(k)}$ and $g^{(k)}$, need to be maintained during the solution process and no explicit matrices or tensors are constructed or stored. Note that if $p$ is even, \cref{alg:pm} reduces to the power method for computing the largest eigenvalues of nonnegative tensors if one recalls \cref{prop:tensor} for equivalence of graph $p$-Laplacian and tensors. To conclude this section, we give some numerical experiments that demonstrate the exceptional robustness and speed of our algorithm. In the following examples, we assume $w\equiv1$, $\mu\equiv1$, $\kappa\equiv0$, and $\sigma\equiv-1$.

\begin{example}
  We first solve the largest eigenvalue of signless $p$-Laplacian of $K_8 \vee \overline{K}_{12}$ with $p=\frac{10}{3}$. Twenty random nonnegative initial vector $f^{(0)}$ are used and the relative gaps between $\overline{\lambda}_{k}$ and $\underline{\lambda}_{k}$ during iterations are shown in the left panel of \cref{fig:ex4}. The curves are similar in that they decay monotonically until approximately the 55th iteration and level off from then on, which shows that our algorithm is insensitive to the initial vector. The constant slopes of these curves also suggest linear convergence rate for our algorithm, which is analogous to the results proved in \cite{ZQ} for even $p$. In fact, though lacking a rigorous proof, the linear convergence rate is observed for both integer $p$ and rational $p$.
\end{example}

\begin{example}
  Next, we generate 1000 random graphs and compute their largest eigenvalue of signless 20-Laplacian. These graphs have a fixed number of 1000 vertices but differ in the number of edges. According to the definition of $p$-Laplacian, the complexity of our algorithm is $\mathcal{O}(p|E|)$, where $|E|$ denotes the number of edges. This is confirmed by the numerical results from the right panel of \cref{fig:ex4}, where the solution time is proportional to $|E|$. Note that it takes less than 0.4 seconds to compute the largest eigenvalue of signless 20-Laplacian for a graph with 1000 vertices and 250000 edges. In contrast, at least half a second is required to compute the largest eigenvalue using the algorithm in \cite{CDN} for solving the ordinary problem in \cref{ex:nodal}.
\end{example}

\begin{figure}[!t]
    \centering
    \subfloat{\label{subfig:ex4rate}
    \includegraphics[width = 6cm]{"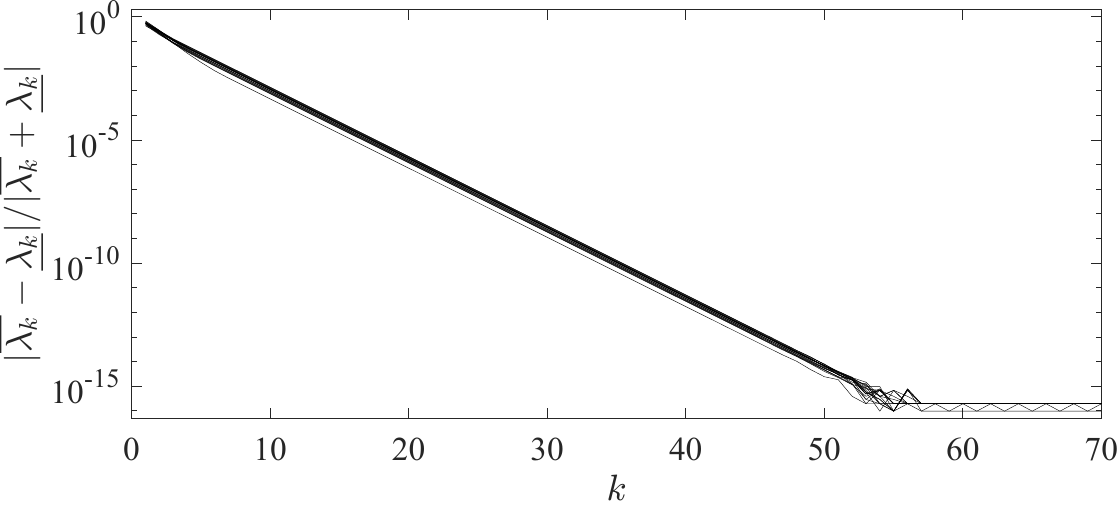"}} \quad
    \subfloat{\label{subfig:ex4time}
    \includegraphics[width = 6cm]{"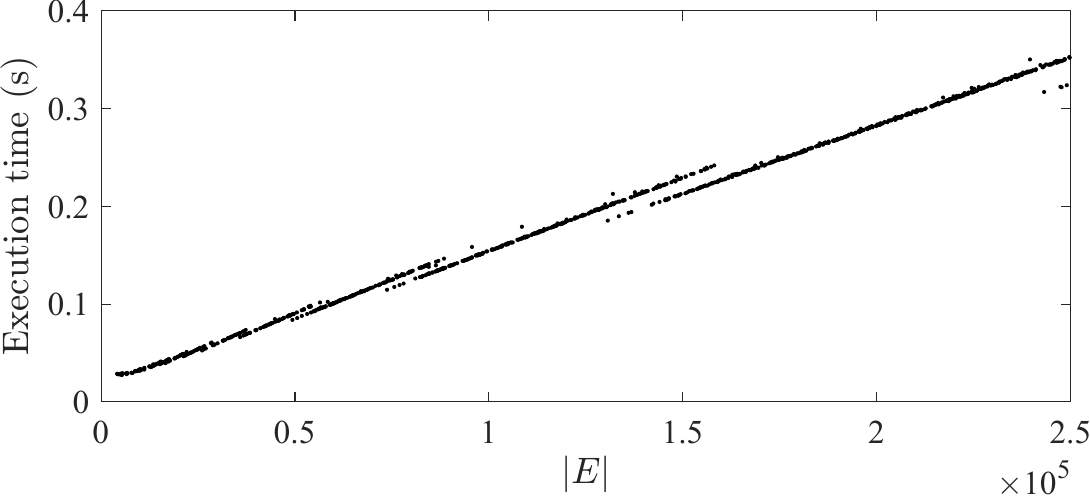"}}
    \caption{Left: relative gap between $\overline{\lambda}_{k}$ and $\underline{\lambda}_{k}$ versus iteration for random initial vectors. Right: time for solving the largest eigenvalue of signless 20-Laplacian of different graphs.}
    \label{fig:ex4}
\end{figure}

\section{Application}\label{sec:app}
In this section, we only consider graphs with edge weight $w\equiv1$, vertex measure $\mu\equiv1$, and potential function $\kappa\equiv0$, since almost all of the graphs considered in these forbidden subgraph type problems are such graphs. So we denote graph $G=(V,E)$ instead of  $G=(V,E,w,\mu,\kappa)$ for short. We use $\lambda^{(p)}_{\max}(G)$ to denote the largest eigenvalue of signless $p$-Laplacian of $G$ in this section.

We first propose a criterion to determine whether graph $G'$ can not be a subgraph of graph $G$ (see \cref{Criterion}), and verify from the two perspectives of theory and experiment that this criterion is superior to the criterion in the linear case.
\begin{lemma}\label{lemma:criterion}
  If $G'=(V',E')$ is a subgraph of $G=(V,E)$, then  $\lambda^{(p)}_{\max}(G')\leq \lambda^{(p)}_{\max}(G)$ for any $p>1$.
\end{lemma}
\begin{proof}
  Let $f$ be the eigenfunction of signless $p$-Laplacian of $G'$ corresponding to $\lambda^{(p)}_{\max}(G')$ with $\sum_{i\in V'}|f_i|^p=1$. By definition, we have $$\lambda_{\max}^{(p)}(G')=\frac{\sum_{\{i,j\}\in E'}w_{ij}\left|f_i+f_j\right|^p}{\sum_{i\in V'}|f_i|^p}.$$
 We define $g\in C(V)$ by $g(i)=f(i)$ if $i\in V'$ and $g(i)=0$ if $i\notin V'$. Since $\sum_{i\in V}|g_i|^p=1$, we directly compute
  \begin{align*}
    \begin{aligned}
      \lambda_{\max}^{(p)}(G)=&\max_{h\in \mathcal{S}_p(V)}\frac{\sum_{\{i,j\}\in E}w_{ij}\left|h_i+h_j\right|^p}{\sum_{i\in V}|h_i|^p}\geq \frac{\sum_{\{i,j\}\in E}w_{ij}\left|g_i+g_j\right|^p}{\sum_{i\in V}|g_i|^p}\\
      &\geq \frac{\sum_{\{i,j\}\in E'}w_{ij}\left|f_i+f_j\right|^p}{\sum_{i\in V'}|f_i|^p}=\lambda_{\max}^{(p)}(G').
    \end{aligned}
  \end{align*} 
  This concludes the proof.
\end{proof}

By the above lemma, we give the following criterion.
\begin{criterion}\label{Criterion}
  If $\lambda^{(p)}_{\max}(G')>\lambda^{(p)}_{\max}(G)$ for some $p>1$, then $G'$ can not be a subgraph of $G$.
\end{criterion}

Next, we present the following theorem to demonstrate the advantages of our criterion.
\begin{theorem}
  Let $G=(V,E)$ be a graph. We have $\lambda_{\max}^{(p)}(K_{1,d})\leq \lambda_{\max}^{(p)}(G)$ for any $p>1$ if and only if $K_{1,d}$ is a subgraph of $G$.
\end{theorem}
\begin{proof}
  By \cref{lemma:criterion}, we have that if $K_{1,d}$ is a subgraph of $G$, then $\lambda_{\max}^{(p)}(K_{1,d})\leq \lambda_{\max}^{(p)}(G)$. Next, we assume $\lambda_{\max}^{(p)}(K_{1,d})\leq \lambda_{\max}^{(p)}(G)$ and $G$ does not contain $K_{1,d}$. Then the maximal degree of $G$ is less than or equal to $d-1$. By Lemma 10 in \cite{BS18} , we have $\lambda_{\max}^{(p)}(G)\leq 2^{p-1}(d-1)$. We do the similar computation as Lemma 3.1 in \cite{amghibech2006bounds} to get $\lambda_{\max}^{(p)}(K_{1,d})=(1+d^{\frac{1}{p-1}})^{p-1}$. We have $\limsup_{p\to 1^+}\lambda^{(p)}_{\max}(G)\leq d-1$ and $\lim_{p\to 1^+}\lambda^{(p)}_{\max}(K_{1,d})=d$. So there must exist $\hat{p}>1$ that is sufficiently close to 1 such that  $\lambda^{(\hat{p})}_{\max}(G)<\lambda^{(\hat{p})}_{\max}(K_{1,d})$. Contradiction.
\end{proof}

At the end of this section, we give an example for which the usual criterion based on linear Laplacian and adjacency matrix fails, while \cref{Criterion} remains effective for some $1<\hat{p}<2$.
\begin{example}
  Consider $G$ plotted in the left panel of \cref{fig:ex5} and $G' = K_3 \vee \overline{K_3}$. We compute their largest eigenvalues of signless $p$-Laplacian (scaled by $2^{-p}$) with different $p$ varying from 1.01 to 5. The results are shown in the right panel of \cref{fig:ex5}. On the right of the vertical dotted-dashed line, i.e. $p \geq 2$, it is obvious that $\lambda^{(p)}_{\max}(G') < \lambda^{(p)}_{\max}(G)$. However, when $p$ approaches 1, the relation is reversed and we immediately know that $G'$ can not be a subgraph of $G$ according to \cref{Criterion}. Additionally, one can easily compute the largest eigenvalues of signless Laplacian, Laplacian, and adjacency matrix of $G'$ and $G$ and observe that all eigenvalues associated with $G'$ are no greater than those associated with $G$, thereby indicating the failure of criterion based on linear operators. 
\end{example}

\begin{figure}[!t]
  \centering
  \subfloat{\label{subfig:ex5graph}
  \includegraphics[width = 4cm]{"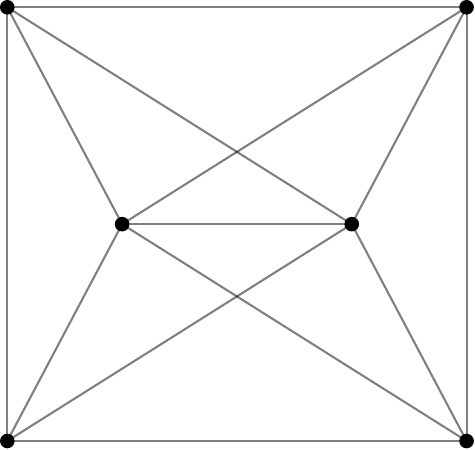"}} \quad
  \subfloat{\label{subfig:ex5eigs}
  \includegraphics[width = 8cm]{"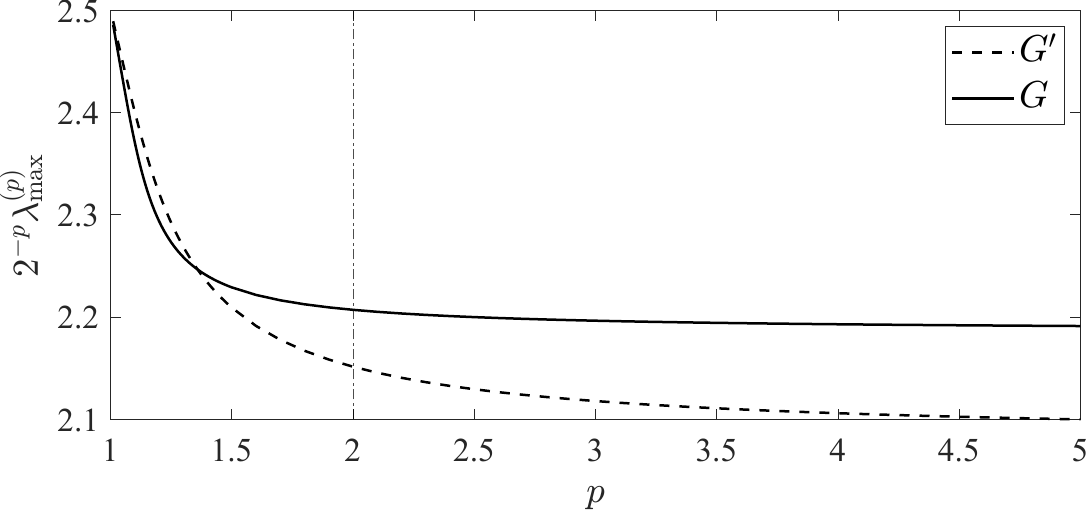"}}
  \caption{Left: the graph $G'$ used in Example 5.1. Right: the largest eigenvalues of signless $p$-Laplacian of $G$ and $G'$ for various $p$.}
  \label{fig:ex5}
\end{figure}

\section{Further remarks}\label{sec:remark}
In this article, we introduce the method of tensor eigenvalues into the graph $p$-Laplacian and provide algorithms for calculating the eigenvalues and eigenfunctions of signed graph $p$-Laplacian. 

When $p\geq 2$ is even, the graph $p$-Laplacian eigenproblems is equivalent to the tensor eigenproblems. However, we found that for any $p>1$, although these two types of problems can not be equivalent to each other, they have many similarities. Therefore, we firmly believe that there are many other tensor analysis results and methods that can be extended to the graph $p$-Laplacian eigenproblems even if $p>1$ is not even.

Currently, the tensor analysis is much more developed than graph $p$-Laplacian, so we hope that more tensor techniques can be applied to graph $p$-Laplacian, solving problems in the latter. On the other hand, the graph $p$-Laplacian has some additional results due to its graph properties, which may have potential applications in tensor analysis.

Considering the above situation, we propose the following open-ended question.
\begin{question}
  What other results in tensor analysis besides this article can be extended to graph $p$-Laplacian? What results of graph $p$-Laplacian can be used to improve results in tensor analysis? 
\end{question}

It is proved that the $\frac{\lambda^{(p)}_{\max}(G)}{2^p}$ convergent to the largest eigenvalue of the adjacency matrix of $G$ \cite{GLZ24}, which is also observed in \cref{fig:ex5}. This implies the following question
\begin{question}
  Given a graph $G=(V,E,w,\mu,\kappa)$, what is the convergence rate of $\frac{\lambda^{(p)}_{\max}(G)}{2^p}$?
\end{question}

Moreover, we see in \cref{fig:ex5} that $\frac{\lambda^{(p)}_{\max}(G)}{2^p}$ is convex and this holds for all the graphs we calculated. Thus, another question arises.
\begin{question}
  It is true that the curve $\frac{\lambda^{(p)}_{\max}(G)}{2^p}$ is convex for any graph $G$?
\end{question}

Ultimately, a challenging problem lies in determining the convergence rate of \cref{alg:pm}. The cases for even $p$ can be addressed through the analysis of tensor eigenvalue problems \cite{ZQ}, while the remaining cases for general $p$ remain an open question to be explored.
\begin{question}
  Can we give the proof of linear convergence rate of the algorithm in \cref{alg:pm} when $p$ is not an even integer?
\end{question}

\section*{Acknowledgments}
We are very grateful to Dong Zhang for his inspiring discussions and Yuchuan Zheng for his support with computing resources. We also appreciate Prof. Cui for generously providing the codes used in \cite{CDN}.

\bibliographystyle{siamplain}
\bibliography{references}
\end{document}